\documentclass[10pt, regno]{article}
\usepackage{amsmath,amssymb,amsfonts}
\usepackage{amsthm}
\usepackage{geometry}
\usepackage{enumitem}
\usepackage{subcaption}
\usepackage{bbm}
\usepackage{mathtools} 
\usepackage{graphicx}
\usepackage{tikz}
\usepackage{titling}

\definecolor{applegreen}{rgb}{0.55, 0.71, 0.0}
\definecolor{bleudefrance}{rgb}{0.19, 0.55, 0.91}

\usepackage{hyperref}
\hypersetup{colorlinks=true,linkcolor=bleudefrance,citecolor=applegreen}

\usepackage{cleveref}
\usepackage{setspace}
\usepackage{fullpage}
\usepackage{microtype}

\theoremstyle{plain}
\newtheorem{thm}{Theorem}[section]
\newtheorem{lem}[thm]{Lemma}

\newtheorem{cor}[thm]{Corollary}

\theoremstyle{definition}

\newtheorem*{rk}{Remark}

\newcommand{\p}{\mathbb{P}}

\newcommand{\arb}{\mathbb{A}}
\newcommand{\warb}{\mathbb{W}}

\newcommand{\abs}[1]{\left\lvert #1 \right\rvert}
\newcommand{\bra}[1]{\left( #1 \right)}

\makeatletter
\pgfdeclareshape{slash underlined}
{
  \inheritsavedanchors[from=rectangle] % this is nearly a circle
  \inheritanchorborder[from=rectangle]
  \inheritanchor[from=rectangle]{north}
  \inheritanchor[from=rectangle]{north west}
  \inheritanchor[from=rectangle]{north east}
  \inheritanchor[from=rectangle]{center}
  \inheritanchor[from=rectangle]{west}
  \inheritanchor[from=rectangle]{east}
  \inheritanchor[from=rectangle]{mid}
  \inheritanchor[from=rectangle]{mid west}
  \inheritanchor[from=rectangle]{mid east}
  \inheritanchor[from=rectangle]{base}
  \inheritanchor[from=rectangle]{base west}
  \inheritanchor[from=rectangle]{base east}
  \inheritanchor[from=rectangle]{south}
  \inheritanchor[from=rectangle]{south west}
  \inheritanchor[from=rectangle]{south east}
  \inheritanchorborder[from=rectangle]
  \foregroundpath{
    \southwest \pgf@xa=\pgf@x \pgf@ya=\pgf@y
    \northeast \pgf@xb=\pgf@x \pgf@yb=\pgf@y
    \pgf@xc=\pgf@xa
    \advance\pgf@xc by .5\pgf@xb
    \pgf@yc=\pgf@ya
    \advance\pgf@xc by -1.3pt
    \advance\pgf@yc by -1.8pt
    \pgfpathmoveto{\pgfqpoint{\pgf@xc}{\pgf@yc}}
    \advance\pgf@xc by  2.6pt
    \advance\pgf@yc by  3.6pt
    \pgfpathlineto{\pgfqpoint{\pgf@xc}{\pgf@yc}}
    \pgfpathmoveto{\pgfqpoint{\pgf@xa}{\pgf@ya}}
    \pgfpathlineto{\pgfqpoint{\pgf@xb}{\pgf@ya}}
 }
}
\tikzset{nomorepostaction/.code=\let\tikz@postactions\pgfutil@empty}

\newcommand\nxleftrightarrow[2][]{%
  \mathrel{\tikz[baseline=-.7ex] \path node[slash underlined,draw,<->,anchor=south] {\(\scriptstyle #2\)} node[anchor=north] {\(\scriptstyle #1\)};}}
\makeatother

\title{The wired arboreal gas on regular trees}
\author{Philip Easo \\ \normalsize{DPMMS, University of Cambridge} \\ \normalsize{\href{mailto:pe265@cam.ac.uk}{pe265@cam.ac.uk}}}

\date{\small{\today}}

\renewenvironment{abstract}
 {\small\par\noindent\textbf{\abstractname.}\ }
 {\par\medskip}

\begin{document}
\maketitle

\begin{abstract}
We study the weak limit of the arboreal gas along any exhaustion of a regular tree with wired boundary conditions. We prove that this limit exists, does not depend on the choice of exhaustion, and undergoes a phase transition. Below and at criticality, we prove the model is equivalent to bond percolation. Above criticality, we characterise the model as the superposition of critical bond percolation and a random collection of infinite one-ended paths. This provides a simple example of an arboreal gas model that continues to exhibit critical-like behaviour throughout its supercritical phase.

\end{abstract}

\section{Introduction}
Let $G=(V,E)$ be a finite graph. The arboreal gas is a model for a random spanning subgraph $\omega \in \{ 0,1 \}^E$ that is a \emph{forest}, a graph without cycles. Given a parameter $\beta > 0$, the arboreal gas $\arb_\beta^G$ assigns weight $\arb_\beta^G \bra{ \omega = \eta } = \beta^{\abs{ \eta }} / Z_\beta$ to each forest $\eta \in \{ 0,1 \}^E$ with $\abs{\eta}$ edges, where $Z_\beta$ is the suitable normalising constant. In terms of bond percolation $\p_p^G$, in which every edge in $E$ is independently included in $\omega$ with probability $p \in [0,1]$,
\begin{equation} \label{eq:arb-perc}
	\arb_\beta^G \bra{ \; \cdot \;} = \p_{p_\beta}^G \bra{ \; \cdot \; \mid \omega \text{ is a forest} } \quad \text{with} \quad p_\beta := \frac{\beta}{\beta+1}.
\end{equation}
This model is at the intersection of a number of other important models in discrete probability: the arboreal gas is the weak limit of the $q$-state random cluster model $\phi_{p,q}$ as $q \to 0$ with $p = \beta q$ \cite[Section 1.5]{random-cluster-book}, and taking $\beta = 1$ and $\beta \to \infty$, the arboreal gas becomes the uniform measure on spanning forests and the uniform spanning tree, respectively.

The aboreal gas was studied on the complete graph in \cite{complete1,complete2}. In this setting, the aboreal gas has a phase transition at the same point as the corresponding bond percolation process, and the two models are equivalent in the subcritical phase. However, in the supercritical phase, the following surprising phenomenon occurs: the second largest cluster in the arboreal gas scales like $n^{2/3}$, as in the critical arboreal gas (and critical bond percolation \cite{erdos-renyi1,erdos-renyi2}), whereas the second largest cluster in bond percolation scales like $\log n$, as in subcritical bond percolation. More recently, weak limits of the arboreal gas along exhaustions of $\mathbb Z^d$ were studied in \cite{lattice1,lattice2} by exactly relating the aboreal gas to a non-linear sigma model with hyperbolic target space. Here again, the authors find that in the supercritical phase, the arboreal gas continues to exhibit critical-like behaviour, this time with respect to the  decay of certain truncated two-point functions. They remark that this phenomenon is actually natural from the viewpoint of spontaneously broken continuous symmetries, rather than bond percolation.

The motivation for this paper is to give a transparent example of an arboreal gas model exhibiting this behaviour and to establish it by elementary arguments. Let $T$ be the $k$-regular tree for some $k \geq 3$. Since $T$ is infinite, the arboreal gas on $T$ is not defined a priori. To extend the definition of the arboreal gas to $T$, we take the weak limit along an \emph{exhaustion} $(V_n)$. An exhaustion of $T$ is a growing sequence of finite connected sets of vertices whose union contains every vertex in $T$. For each $n$, construct a finite graph $G_n$ as follows: take the subgraph of $T$ induced by the union of $V_n$ with its neighbours in $T$, then merge these neighbours to a single vertex $\partial_n$. This corresponds to equipping $(V_n)$ with wired boundary conditions. If we used free boundary conditions instead, i.e.\ we did not merge the boundary vertices, then each $G_n$ would be a tree, so the arboreal gas on $G_n$ would simply be bond percolation. The wired arboreal gas on $T$ of parameter $\beta > 0$, denoted $\warb_\beta^T$, is defined to be the weak limit of $\arb_\beta^{G_n}$ as $n \to \infty$.

The goal of this paper is to prove the following theorem characterising $\warb_\beta^T$. Notice $\warb_\beta^T$ undergoes a phase transition at $\beta_c := \frac{1}{k-2}$. This corresponds to the critical parameter $p_c := \frac{1}{k-1}$ for bond percolation in the sense that $p_c = p_{\beta_c}$. In the subcritical and critical phases, $\warb_\beta^T$ is equal to $\p_{p_\beta}^T$, whereas in the supercritical phase, $\warb_\beta^T$ is the superposition of $\p_{p_c}^T$ and a random collection of infinite one-ended paths. In particular, the supercritical clusters are distributed as in critical bond percolation when finite and as the incipient infinite cluster \cite{iic} when infinite. We write $o$ for an arbitrary distinguished vertex in $T$, we write $d$ for the graph distance metric on $T$, and we label the endpoints of each edge $e \in T$ by $e^-$ and $e^+$ with $d(o,e^-) < d(o,e^+)$.

\begin{thm} \label{thm}
	Let $(G_n)$ be the sequence of graphs induced by an exhaustion of a $k$-regular tree $T$ with $k\geq 3$ using wired boundary conditions, and let $\beta >0$.
	\begin{enumerate}[topsep=0pt]
		\item The sequence $(\arb_\beta^{G_n} )_{n =1}^\infty$ converges weakly to a limit $\warb_\beta^T$ that is independent of the choice of exhaustion and is therefore invariant under any graph automorphism of $T$.
		\item When $\beta \leq \frac{1}{k-2}$, we have $\warb_\beta^{T} = \p_{p_\beta}^{T}$.
		\item When $\beta \geq \frac{1}{k-2}$, we can sample $\omega$ according to $\warb_\beta^T$ by the following procedure:
		\begin{enumerate}
			\item Sample a configuration $\omega_0$ on the edges of $T$ according to $\p_{p_c}^T$.
			\item Sample a configuration $\eta$ on the vertices of $T$ such that each vertex $v$ is independently included in $\eta$ with probability $\frac{\beta(k-2)k-k}{\beta(k-2)k-1}$ if $v = o$ and probability $\frac{\beta(k-2)-1}{\beta(k-2)}$ if $v \not=o$.
			\item Let $U$ be the set of vertices $v \in \eta$ such that $v = o$ or there is an edge $e \in T \backslash \omega_0$ with $ e^+ = v$.
			\item For each vertex $v \in T$, independently select a neighbour $s(v)$ with $d(o,s(v)) > d(o,v)$ uniformly at random. Let $\gamma(v)$ denote the edges in the path $v, s(v), s(s(v)), \dots$.
			\item Set
			\[
				\omega = \omega_0 \cup \bigcup_{v \in U} \gamma(v).
			\]
		\end{enumerate}
	\end{enumerate}
\end{thm}

The wired uniform spanning forest, i.e.\ the weak limit of the uniform spanning tree with wired boundary conditions, can be sampled by the above procedure if we instead insist that $\eta$ includes every vertex \cite{WUSF}. So as $\beta \to \infty$, the wired arboreal gas $\warb_\beta^T$ converges weakly to the wired uniform spanning forest, which we will denote $\warb_{\infty}^T$. This is an infinite analogue of the fact that on finite graphs, the arboreal gas converges weakly to the uniform spanning tree as $\beta \to \infty$.

The next corollary contains two stochastic domination properties of the wired arboreal gas. We prove these by using the procedure from \cref{thm} to build monotone couplings. For item 1, we use the fact that there is a monotone coupling of the bond percolation measures $(\p_p^T)_{p \in [0,1]}$ \cite[Theorem 2.1]{percolation-book} to build a monotone coupling of the measures $(\warb_{\beta}^T)_{\beta >0}$. For item 2, we modify the procedure in order to sample from the wired arboreal gas conditioned on the state of a given edge. Item 2 is a negative dependence property that is weaker than negative association but stronger than edge-negative correlation. The arboreal gas may satisfy negative association in general, but even edge-negative correlation has not been proved. This is the main obstacle to defining the arboreal gas on arbitrary infinite graphs, including $\mathbb Z^d$ \cite{lattice2}.
 
\begin{cor}[Stochastic Domination] \label{cor:stochastic-domination} 
	\hfill
	\begin{enumerate}[topsep=0pt]
		\item For all $\beta_1,\beta_2 \in (0,\infty]$ with $\beta_1 < \beta_2$, we have $\warb_{\beta_1}^T \lesssim \warb_{\beta_2}^T$.
		\item For each edge $e \in T$, we have $\warb_\beta^T \bra{ \; \cdot \; \mid e \in \omega } \lesssim \warb_\beta^T \bra{ \; \cdot \; \mid e \not\in \omega } $, where these are viewed as measures on the configurations of the edges in $T \backslash \{ e \}$.
	\end{enumerate}
\end{cor}

\begin{rk}
	While writing this paper, we learned that G. Ray and B. Xiao have also been independently studying the wired arboreal gas on regular trees \cite{forests-on-wired-trees}. Their main result is similar to ours. However, their arguments have a slightly different flavour, being more similar to \cite{random-cluster-on-tree}.
\end{rk}

\section{Existence via cylinder events}
In this section, we prove that $\warb_\beta^T$ is well-defined and compute the probability it assigns to certain cylinder events, i.e.\  events that depend on finitely many edges. We naturally identify the edges in each $G_n$ with edges in $T$, so that for each edge $e \in G_n$, the endpoints $e^-$ and $e^+$ are defined as vertices in $G_n$. Given vertices $u,v \in T$, we say $u$ is a descendant of $v$ if the geodesic in $T$ from $o$ to $u$ crosses $v$. For each $G_n$ and edge $e \in G_n$, we define $G_n(e)$ to be the subgraph of $G_n$ induced by $\{e^-,\partial_n\}$ and the descendants of $e^+$. Finally, we drop the subscripts in $\arb_\beta^G, \warb_\beta^G, \p_{p_\beta}^G, p_\beta$ whenever this does not cause confusion.

\begin{lem}
	Let $e_1,\dots ,e_r$ be edges in $T$ such that $e_i^-$ is not a descendent of $e_j^-$ for any $i \not=j$. Pick $n$ sufficiently large that $e_1^{-},\dots ,e_r^- \in V_n$, and let $H$ be the graph formed from $G_n(e_1),\dots, G_n(e_r)$ by identifying $e_1^-, \dots ,e_r^-$ to a single vertex $v$. For each $i$, define $q_i := \arb^{G_n(e_i)} \bra{ e_i^- \not \leftrightarrow \partial_n }$. Then
	\begin{equation} \label{eq:discon-formula}
			\arb^H \bra{ v \not \leftrightarrow \partial_n } = \frac{1}{1-r + \sum_i \frac{1}{q_i}},
	\end{equation}
	and abbreviating $\{ \omega \text{ is a forest} \}$ to $\{ \text{forest} \}$,
	\begin{equation} \label{eq:forest-formula}
			\p^H   \bra{ \text{forest}} = \bra{ \prod_i q_i + \sum_i (1-q_i) \prod_{j \not= i} q_j } \cdot \prod_i \p^{G_n(e_i)} \bra{ \text{forest}}.
	\end{equation}
\end{lem}

\begin{proof}
	Every cycle in $H$ is either entirely contained in some $G_n(e_i)$ or consists of two edge-disjoint paths from $v$ to $\partial_n$ that are each entirely contained in some distinct $G_n(e_i)$ and $G_n(e_j)$, respectively. So by independence and \eqref{eq:arb-perc},
	\begin{equation} \label{eq:cases-discon} \begin{split}
		\p^H \bra{ \{ \text{forest} \} \cap \{ v \not \leftrightarrow \partial_n \} } &= \prod_i \p^{G_n(e_i)} \bra{ \{ \text{forest} \} \cap \{ e_i^- \not\leftrightarrow \partial_n \} } \\ &= \bra{ \prod_i q_i } \cdot \prod_i \p^{G_n(e_i)} \bra{\text{forest}},
	\end{split} \end{equation}
	\begin{equation} \label{eq:cases-con} \begin{split}
	\p^H \bra{ \{ \text{forest} \} \cap \{v \leftrightarrow \partial_n \} } &= \sum_i \p^{G_n(e_i)}	\bra{ \{ \text{forest} \} \cap \{e_i^- \leftrightarrow \partial_n \} } \prod_{j \not= i} \p^{G_n(e_j)} \bra{ \{ \text{forest} \} \cap \{ e_j^- \not\leftrightarrow  \partial_n\} }\\
	&= \bra{\sum_i (1-q_i) \prod_{j \not=i} q_j } \cdot \prod_i \p^{G_n(e_i)} \bra{\text{forest}}.
	\end{split} \end{equation}
Adding \eqref{eq:cases-discon} to \eqref{eq:cases-con} gives \eqref{eq:forest-formula}. Dividing \eqref{eq:cases-discon} by \eqref{eq:forest-formula} gives an expression for $\arb^H\bra{v \not\leftrightarrow \partial_n}$ that rearranges to the claimed formula \eqref{eq:discon-formula}.
\end{proof}

\begin{lem} \label{lem:limit}
For each edge $e \in T$,
\[
	\lim_{n \to \infty} \arb^{G_n(e)} \bra{ e^- \not\leftrightarrow \partial_n } = \lambda :=
	\begin{cases}
		1 \quad & \text{if } \beta \leq \frac{1}{k-2},\\
		\frac{k-1}{(k-2)(1+\beta)} \quad & \text{if } \beta \geq \frac{1}{k-2}.
	\end{cases}
\]
\end{lem}

\begin{proof}
This formula will follow from the fact that $\lambda$ is the smallest fixed point of a suitable recursion function $F$.
Pick $n$ sufficiently large that $e^+ \in V_n$. Consider an edge $f \in G_n(e)$ with $f^+ \not= \partial_n$. There are $k-1$ edges $f_1,\dots ,f_{k-1} \in G_n(e)$ with $f_i^- = f^+$. For each $i$, define $q_i := \arb^{G_n(f_i)} \bra{f_i^- \not \leftrightarrow \partial_n }$. By \eqref{eq:discon-formula} and conditioning on the state of $f$,
\begin{equation} \label{eq:discon-with-extend}
	\arb^{G_n(f)} \bra{ f^- \not \leftrightarrow \partial_n } = (1-p) + p \cdot \frac{1}{1-(k-1)+\sum_i \frac{1}{q_i}}. 
\end{equation}
Define a function $F:(0,1] \to (0,1]$ by $F(q) := (1-p) + p \cdot  \frac{1}{1-(k-1)+(k-1)/q}$, which comes from setting $q_1=\dots = q_{k-1}=q$ in \eqref{eq:discon-with-extend}. Notice that $\lambda$ is a fixed point of $F$, and for every $q \in (0,\lambda)$, we have $F(q) \in(q,\lambda)$. In particular, if $q_1,\dots, q_{k-1} \in (0,\lambda)$, then
\begin{equation} \label{eq:f-inequality}
	\min_i q_i < F\bra{\min_i q_i} \leq \arb^{G_n(f)} \bra{ f^- \not \leftrightarrow \partial_n } \leq F\bra{\max_i q_i} < \lambda.
\end{equation}
For every edge $f \in G_n(e)$ with $f^+ = \partial_n$, we trivially have $\arb^{G_n(f)} \bra{ f^- \not\leftrightarrow \partial_n} = 1-p \in (0,\lambda)$. So by \cref{eq:f-inequality} and induction on the graph distance from $\partial_n$, we have $\arb^{G_n(f)} \bra{ f^-  \not \leftrightarrow \partial_n} \in (1-p,\lambda)$ for every edge $f \in G_n(e)$.

Let $R_n$ be the maximum integer such that the ball of vertices $\{ u \in T : d(e^-,u) \leq R_n \}$ is contained in $V_n$. We have shown that for every edge $f \in G_n(e)$, in particular every edge $f \in G_n(e)$ with $d(e^-,f^-) = R_n$, we have $\arb^{G_n(f)} \bra{ f^-  \not \leftrightarrow \partial_n} \in (1-p,\lambda)$. So by \eqref{eq:f-inequality} and induction on the value of $R_n$,
\begin{equation} \label{eq:squeeze}
	\underbrace{F \circ F \circ \dots \circ F}_{R_n \text{ copies}}(1-p) \leq \arb^{G_n(e)} \bra{ e^- \not \leftrightarrow \partial_n } < \lambda.
\end{equation}
By \eqref{eq:f-inequality}, the sequence $(1-p)$, $F(1-p)$, $F \circ F (1-p)$, $F \circ F \circ F (1-p)$, $\dots$ is an increasing sequence of real numbers in $(0,\lambda)$. So this sequence converges to some limit $l \in (0,\lambda]$. Since $F$ is continuous, $l$ must be a fixed point of $F$. Since $F(q) > q$ whenever $q \in (0,\lambda)$, we know $l=\lambda$. In particular, since $R_n \to \infty$ as $n \to \infty$,
\[
	\lim_{n \to \infty} \underbrace{F \circ F \circ \dots \circ F}_{R_n \text{ copies}}(1-p) = \lambda.
\]
The result now follows from \eqref{eq:squeeze}.
\end{proof}

We use these lemmas to relate the limiting probabilities of certain cylinder events under $\arb^{G_n}$ to their probabilities under $\p^T$. Let $B$ be a finite connected set of edges in $T$ that contains an edge adjacent to $o$, and let $\eta \in \{ 0,1 \}^B$ be a configuration on $B$. For every set of edges $E \subseteq T$, we write $\partial E$ for the set of edges in $T \backslash E$ that are adjacent to $E$. $\eta$ induces the equivalence relation on $\partial B$ in which edges $e$ and $f$ are related if and only if $e^-$ and $f^-$ are connected by a path in $\eta$. The next lemma explains how to use the sizes of the equivalence classes $A_1,\dots,A_t$ to compute the limiting probability of the cylinder event $\{ \omega \cap B = \eta \}$ under $\arb^{G_n}$.

\begin{lem} \label{lem:big-formula} For each non-negative integer $m$, define $Q_m := \lambda^m + m(1-\lambda) \lambda^{m-1}$, where $\lambda$ is the limit from \cref{lem:limit}. Then
\[
	\lim_{n \to \infty} \arb^{G_n} \bra{ \omega \cap {B} = \eta } = \frac{\prod_{i=1}^t Q_{\abs{A_i}}}{Q_k \cdot Q_{k-1}^{\abs{ B}}} \cdot \p^{T} \bra{ \omega \cap {B} = \eta}.
\]
\end{lem}

\begin{proof}
	Pick $n$ sufficiently large that $G_n$ contains $B \cup \partial B$. By \eqref{eq:arb-perc},
	\begin{equation} \label{eq:rearrange-conditional}
	\arb^{G_n} \bra{ \omega \cap B = \eta } = \p^{G_n} \bra{ \omega \cap B = \eta  \mid \text{forest}} = \frac{\p^{G_n} \bra{ \text{forest} \mid \omega \cap B = \eta} }{\p^{G_n} \bra{ \text{forest} } } \cdot \p^T \bra{ \omega \cap B = \eta }.
	\end{equation}

	For each $i \in \{1 ,\dots, t\}$, let $H_i$ be the graph formed from $\{ G_n(e) : e \in A_i \}$ by identifying the vertices in $\{e^- : e \in A_i\}$ to a single vertex $v_i$. Starting from $G_n$, when we contract every edge in $\eta$ and delete every edge in $B \backslash \eta$, we are left with copies of $H_1,\dots, H_t$ that meet only at $\partial_n$. Each cycle in this graph is entirely contained in a copy of some $H_i$. So by independence, $
		\p^{G_n} \bra{ \text{forest} \mid \omega \cap B = \eta } = \prod_{i=1}^t \p^{H_i} \bra{ \text{forest} }$. For each edge $e \in G_n$, define $q_e:= \arb^{G_n(e)} \bra{ e^- \not \leftrightarrow \partial_n }$. By \eqref{eq:forest-formula} and \Cref{lem:limit}, it follows that
		\begin{equation} \label{eq:after-contraction} \begin{split}
			\p^{G_n} \bra{ \text{forest} \mid \omega \cap B = \eta } &= \prod_{i=1}^t \bra{ \prod_{e \in A_i} q_e + \sum_{e \in A_i} (1-q_e) \prod_{f \in A_i \backslash \{ e \}} q_f } \bra{ \prod_{e \in A_i} \p^{G_n(e)} \bra{\text{forest}}} \\
			&\sim \bra{\prod_{i=1}^t Q_{\abs{A_i}}} \cdot \prod_{e \in \partial B} \p^{G_n(e)} \bra{\text{forest}}.
		\end{split} \end{equation}
		Here we use the notation $y(n) \sim z(n)$ to mean $\lim_{n \to \infty} y(n) / z(n) = 1$. Similarly, by \eqref{eq:forest-formula} and \Cref{lem:limit},
		\[ \begin{split}
			\p^{G_n} \bra{ \text{forest} } &= \bra{ \prod_{e : \; e^- = o} q_e + \sum_{e: \; e^- =o} (1-q_e) \prod_{\substack{ f : \; f^- =o \\ f \not= e}} q_f } \cdot \prod_{e :\; e^- = o} \p^{G_n(e)} \bra{ \text{forest} } \\
			&\sim Q_k \cdot \prod_{e :\; e^- = o} \p^{G_n(e)} \bra{ \text{forest} }.
		\end{split} \]
For every edge $e \in G_n$, we have $\p^{G_n(e)} \bra{\text{forest}} = \p^{G_n(e) \backslash \{e\}} \bra{\text{forest}}$, since no cycle in $G_n(e)$ crosses $e^-$. So by repeatedly applying \eqref{eq:forest-formula} and \Cref{lem:limit},
\begin{equation} \label{eq:before-contraction} \begin{split}
	\p^{G_n} \bra{ \text{forest} } &\sim Q_k \cdot \bra{ \prod_{e_1 \in \partial B: \; e_1^- = o} \p^{G_n(e_1)} \bra{ \text{forest}}} \cdot \bra{ \prod_{e_1 \in B:\; e_1^- = o} Q_{k-1} \cdot  \bra{ \prod_{e_2 : \; e_2 ^- = e_1^+} \p^{G_n (e_2)} \bra{ \text{forest} }}  } \\
	&\dots \\
	&\sim Q_k Q_{k-1}^{\abs{B}} \cdot \prod_{e \in \partial B} \p^{G_n(e)} \bra{ \text{forest} }.
\end{split} \end{equation}
The result now follows by plugging \eqref{eq:after-contraction} and \eqref{eq:before-contraction} into \eqref{eq:rearrange-conditional}.
\end{proof}

From \Cref{lem:big-formula}, whose formula does not depend on $(V_n)$, we deduce that $\arb^{G_n}$ converges weakly as $n \to \infty$ to a limit $\warb^T$ that is independent of the choice of exhaustion. This verifies item 1 of \cref{thm}. To prove item 2, take $\beta \leq \frac{1}{k-2}$. By \Cref{lem:limit}, $\lambda = 1$, so $Q_m = 1$ for every non-negative integer $m$. So $\warb^T$ and $\p^T$ assign the same probability to every cylinder event covered by \Cref{lem:big-formula} and hence to every event.

\section{Supercritical phase}
In this section, we prove item 3 from \Cref{thm}, which characterises the supercritical phase of $\warb^{T}$. Let $E_1 \sqcup E_2$ be a partition of the edges adjacent to $o$ in $T$. Consider the corresponding subgraphs $T_1$ and $T_2$ induced by the union of $E_1$ with the descendants of $\{e^+ : e \in E_1 \}$ and the union of $E_2$ with the descendants of $\{e^+ : e \in E_2 \}$, respectively. Our first step is to relate the restricted configurations $\omega \cap T_1$ and $\omega \cap T_2$ to each other. In particular, writing $K_u$ for the cluster containing $u$, we relate the restricted configuration $\omega \cap T_1$ to the restricted cluster $K_o \cap T_2$ on the event that $K_o \cap T_2$ is finite. Interestingly, we find that when conditioned to be finite, $K_o \cap T_2$ is distributed as it is under critical bond percolation $\p_{p_c}^T$.

\begin{lem} \label{lem:technical}
When $\beta \geq \frac{1}{k-2}$,
\begin{enumerate}[topsep=0pt]
	\item Under $\warb^T \bra{ \; \cdot \; \mid \abs{K_o \cap T_2} < \infty }$, the random variables $\omega \cap T_1$ and $K_o \cap T_2$ are independently distributed, and $K_o \cap T_2$ has the same distribution as it does under $\p_{p_c}^T$;
	\item $\warb^T \bra{ \abs{K_o \cap T_2}  < \infty } = \frac{\beta(k-2)\abs{E_1} + \abs{E_2} -1 }{\beta(k-2)k - 1}$.
\end{enumerate}

\end{lem}

\begin{proof}
Let $B_1 \subseteq T_1$ be a finite connected set of edges that contains $E_2$, let $\eta_1 \in \{0,1 \}^{B_1}$ be a configuration on $B_1$, and let $G \subseteq T_2$ be a finite connected subgraph containing $o$. Our first goal is to compute $\warb^T \bra{ \{ \omega \cap {B_1} = \eta_1 \}  \cap \{ K_o \cap T_2 = G \}}$. The event $\{K_o \cap T_2 = G\}$ is the event that the edges in $G$, say $O$, are open, and the edges adjacent to $G$ in $T_2$, say $C$, are closed. So we can write this event as $\{\omega \cap {B_2} = \eta_2\}$, where $B_2 := O \cup C$ and $\eta_2 := O$. In particular, we can write the event $\{ \omega \cap {B_1} = \eta_1 \}  \cap \{ K_o \cap T_2 = G \}$ as $\{\omega \cap B = \eta\}$, where $B:= B_1 \cup B_2$ and $\eta := \eta_1 \cup \eta_2$.

As in the setup for \Cref{lem:big-formula}, the configuration $\eta$ induces the equivalence relation on $\partial B$ in which edges $e$ and $f$ are related if and only if $e^-$ and $f^-$ are connected by a path in $\eta$. None of the edges in $\partial B_1$ are related to any of the edges in $\partial B_2$. Moreover, the edges in $\partial B_2$ are partitioned into $\abs{C}$ equivalence classes, each containing $k-1$ edges. Let $a_1, \dots , a_t$ be the sizes of the equivalence classes containing the edges in $\partial B_1$, including repeats. By \Cref{lem:big-formula},
	\[
		\warb^T \bra{ \omega \cap B = \eta } = \frac{Q_{k-1}^{\abs{C}} \cdot \prod_{i=1}^t Q_{a_i}} {Q_k \cdot Q_{k-1}^{\abs{B}}} \cdot \p^T \bra{ \omega = \eta }.
	\]
By induction on $\abs{O}$, we find $\abs{ C } = \abs{E_2} + (k-2) \cdot \abs{O}$. So we can rewrite the above expression as
\[
	\warb^T \bra{ \omega \cap B = \eta } = \bra{ \frac{ \prod_{i=1}^t Q_{a_i} }{Q_k \cdot Q_{k-1}^{ \abs{B_1} } } \cdot p^{\abs{\eta_1}} ( 1-p)^{\abs{B_1 \backslash \eta_1}+\abs{E_2}}} \cdot \bra{ \frac{p (1-p)^{k-2}}{Q_{k-1}} } ^{\abs{O}}.
\]
By \Cref{lem:big-formula} again, the first term in this product is $\warb^T \bra{ \{\omega \cap B_1 = \eta_1\} \cap \{ E_2 \text{ closed} \} } $. By direct calculation, we miraculously find that the second term in this
 product is $\bra{p_c(1-p_c)^{k-2}}^{\abs{O}}$, which is equal to $\p_{p_c}^T \bra{ K_o \cap T_2 = G } / (1-p_c)^{\abs{E_2}}$. Therefore,
 \begin{equation} \label{eq:unprocessed}
	 \warb^T \bra{ \omega \cap B = \eta } = \frac{\warb^T \bra{ \{ \omega \cap B_1 = \eta_1 \} \cap \{ E_2 \text{ closed} \} } }{(1-p_c)^{\abs{E_2}}} \cdot \p_{p_c}^T \bra{ K_o \cap T_2 = G }.
 \end{equation}
Since $p_c$ is the critical parameter for bond percolation on $T_2$, we know $K_o \cap T_2$ is $\p_{p_c}^T$-almost surely finite \cite[Proposition 5.4]{trees-networks}. Summing \eqref{eq:unprocessed} over all possibilities for $G$, it follows that
\begin{equation} \label{eq:summed}
	\warb^T \bra{ \{\omega \cap B_1 = \eta_1\} \cap \{ \abs{K_o \cap T_2} < \infty \} } = \frac{\warb^T \bra{ \{ \omega \cap B_1 = \eta_1 \} \cap \{ E_2 \text{ closed} \} } }{(1-p_c)^{\abs{E_2}}}.
\end{equation}
Plugging this back into \eqref{eq:unprocessed} and dividing by $\warb^T \bra{ \abs{K_o \cap T_2} < \infty }$ gives
\[
	\warb^T \bra{ \omega \cap B = \eta \mid \abs{ K_o \cap T_2 } < \infty } = \warb^T \bra{ \omega \cap B_1 = \eta_1 \mid \abs{K_o \cap T_2} < \infty } \cdot \p_{p_c}^T \bra{ K_o \cap T_2 = G }.
\]
Since $\eta_1$ and $G$ were arbitrary, this proves item 1. To prove item 2, sum \eqref{eq:summed} over all possibilities for $\eta_1$ given $B_1$ to obtain
\[
	\warb^{T} \bra{ \abs{K_o \cap T_2} < \infty } = \frac{ \warb^T \bra{E_2 \text{ closed}} }{(1-p_c)^{\abs{E_2}}},
\]
then use \Cref{lem:big-formula} to compute
\[
	\frac{ \warb^T \bra{E_2 \text{ closed}} }{(1-p_c)^{\abs{E_2}}} = 
	\frac{Q_{k-\abs{E_2}} \cdot Q_{k-1}^{\abs{E_2}} }{Q_k \cdot Q_{k-1}^{\abs{E_2}} } \cdot \frac{(1-p)^{\abs{E_2}}}{(1-p_c)^{\abs{E_2}}} = \frac{\beta(k-2)\abs{E_1} + \abs{E_2} -1 }{\beta(k-2)k - 1}. \qedhere
\]
\end{proof}

Given a subgraph $H$ of $T$, let $\{o \xleftrightarrow{H} \infty\}$ be the event that $\omega \cap H$ contains an infinite self-avoiding path from $o$. Let $\{ o \xleftrightarrow{H} \infty \} \circ \{o \xleftrightarrow{H} \infty \}$ and $ \{ o \nxleftrightarrow{H} \infty \}  $ be the events that $\omega \cap H$ contains at least two edge-disjoint such paths and no such paths, respectively. When $H = T$, we simply write $\leftrightarrow$ in place of $\xleftrightarrow{H}$. Given an edge $e \in T$, we write $T(e)$ for the subgraph of $T$ induced by $e^-$ and the descendants  of $e^+$.

\begin{cor} \label{cor:finite-infinite} When $\beta \geq \frac{1}{k-2}$,
	\begin{enumerate}[topsep=0pt]
		\item Under $\warb^T \bra{ \; \cdot \; \mid \abs{K_o} < \infty }$, the cluster $K_o$ has the same distribution as it does under $\p_{p_c}^T$;
		\item $\warb^T \bra{ o \leftrightarrow \infty } = \frac{ \beta(k-2)k-k } { \beta(k-2)k-1 };$
		\item $\warb^T \bra{ \{ o \leftrightarrow \infty \} \circ \{ o \leftrightarrow \infty \} } = 0$;
		\item $\warb^T \bra{ o \leftrightarrow \infty \mid o \nxleftrightarrow{T(e)} \infty } = \frac{\beta(k-2)-1}{\beta(k-2)}$ for every edge $e$ adjacent to $o$.
	\end{enumerate}
\end{cor}

\begin{proof}
Items 1 and 2 follow immediately by taking $E_1 = \emptyset$ in items 1 and 2 of \Cref{lem:technical}, respectively. To prove item 3, start by rewriting
\[
	1- \warb^T \bra{ \{ o \leftrightarrow \infty \} \circ \{ o \leftrightarrow \infty \} } = \warb^T \bra{ o \not\leftrightarrow \infty } + \sum_{e: \; e^- = o} \warb^T \bra{ \{ o \xleftrightarrow{ T(e) } \infty \} \cap \bigcap_{\substack{ f: \; f^- = o \\ f \not= e } } {\{ o \nxleftrightarrow{T(f)} \infty \} }}.
\]
By item 2 of \Cref{lem:technical}, for every edge $e$ adjacent to $o$,
\[\begin{split}
	\warb^T \bra{ \{ o \xleftrightarrow{ T(e) } \infty \} \cap \bigcap_{\substack{ f: \; f^- = o \\ f \not= e } } {\{ o \nxleftrightarrow{T(f)} \infty \} }} &= \warb^T \bra{ \bigcap_{\substack{ f: \; f^- = o \\ f \not= e } } {\{ o \nxleftrightarrow{T(f)} \infty \} } } - \warb^T \bra{ o \not\leftrightarrow \infty } \\
	&= \frac{\beta(k-2) + k-2}{\beta(k-2)k-1} - \frac{k-1}{\beta(k-2)k-1} = \frac{\beta(k-2)-1}{\beta(k-2)k-1}.
\end{split}\]
Therefore,
\[
	1-\warb^T \bra{ \{ o \leftrightarrow \infty \} \circ \{ o \leftrightarrow \infty \} } = \frac{k-1}{\beta(k-2)k-1} +k \cdot \frac{\beta(k-2)-1}{\beta(k-2)k-1} = 1.
\]
To prove item 4, given an edge $e$ adjacent to $o$, use the result $\warb^T \bra{ \{ o \leftrightarrow \infty \} \circ \{ o \leftrightarrow \infty \} } = 0$ and \Cref{lem:technical} to compute
\[
	\warb^T \bra{ o \leftrightarrow \infty \mid o \nxleftrightarrow{T(e)} \infty } = \frac{ \warb^T \bra{ o \xleftrightarrow{T \backslash T(e)} \infty } }{ \warb^T \bigg( o \nxleftrightarrow{T(e)} \infty \bigg)} = 
	\frac{\beta(k-2)-1}{\beta(k-2)}. \qedhere
\]
\end{proof}

We now combine our results to prove item 3 of \Cref{thm} and \Cref{cor:stochastic-domination}.

\begin{proof}[Proof of item 3 of \Cref{thm}]
Let $\mathbb Q$ be the distribution of the configuration $\omega$ given by the procedure. For some positive integer $R$, let $B$ be the set of edges $e \in T$ with $d(o,e^+) \leq R$, and for each edge $e \in B$, define $B(e) := T(e) \cap B$. Recall that for every set of edges $E \subseteq T$, we defined $\partial E$ to be the set of edges in $T \backslash E$ that are adjacent to $E$. We write $\partial^2 E$ for $\partial ( E \cup \partial E)$. Define random sets of edges $E_1, E_2, \dots$ by $E_1 := K_o \cap B$, and for each $i \geq 2$, \[E_i := \bigcup_{e \in B \cap \partial^2 E_{i-1}} K(e^-) \cap B(e).\] It suffices to check that for each $i$,
\[
	\mathbb Q \bra{ E_i = \;\cdot \;\mid E_1, \dots , E_{i-1} } = \warb^T \bra{ E_i = \; \cdot \; \mid E_1, \dots , E_{i-1} }.
\]

Start with $i = 1$. By item 2 from \Cref{cor:finite-infinite}, the probability that $K_o$ is finite is the same under $\warb^T$ and $\mathbb Q$. By item 1 from \Cref{cor:finite-infinite}, the conditional distribution of $K_o$ given that $K_o$ is finite is the same under $\warb^T$ and $\mathbb Q$. By item 3 from \Cref{cor:finite-infinite}, on the event that $K_o$ is infinite, there is $\warb^T$-almost surely a unique infinite self-avoiding path $P$ in $\omega$ from $o$. By symmetry, since $\warb^T$ is invariant under any graph automorphism of $T$, this path $P$ is distributed uniformly and hence has the same law as $\gamma(o)$ from the procedure. Let $I$ be the edges of a length-$R$ self-avoiding path from $o$. Let $\mathcal E$ be the event that $K_o$ is infinite and the first $R$ edges in $P$ are the edges in $I$. It suffices to check that the conditional distribution of $(K_o \cap B) \backslash I = \bigcup_{e \in B \cap \partial I} K(e^-) \cap B(e)$ given $\mathcal E$ is the same under $\warb^T$ and $\mathbb Q$. Since $T$ is vertex-transitive, the result $\warb^T \bra{ \{o \leftrightarrow \infty \} \circ \{ o \leftrightarrow \infty \} } = 0$ holds with $o$ replaced by any other given vertex. In particular, on $\mathcal E$, the set $ K(e^-) \cap B(e) $ is finite for every edge $e \in B \cap \partial I$, $\warb^T$-almost surely. So by item 1 from \Cref{lem:technical}, given $\mathcal E$, the sets $\{ K(e^-) \cap B(e) \}_{e \in B \cap \partial I}$ are distributed under $\warb^T$ as they are under $\p_{p_c}^T$ and hence under $\mathbb Q$.

Now take $i \geq 2$. Let $\mathcal A$ be the event that $(E_1, \dots, E_{i-1})$ assumes a particular outcome $(\overline E_1, \dots, \overline E_{i-1})$. On $\mathcal A$, no edge in $B \cap \partial \overline E_{i-1}$ belongs to $\omega$. So by item 1 of \Cref{lem:technical}, given $\mathcal A$, the sets $\{ K(e^+) \cap (B(e) \backslash \{ e \})  \}_{ e \in B \cap \partial \overline E_{i-1} }$ are distributed independently under $\warb^T$, which is also the case under $\mathbb Q$. By item 2 of \Cref{cor:finite-infinite}, given $\mathcal A$, for every edge $e \in B \cap \partial \overline E_{i-1}$, the set $K(e^+) \cap (B(e) \backslash \{ e \})$ has the same probability of being finite under $\warb^T$ and $\mathbb Q$, and arguing similarly to the case $i = 1$, \Cref{lem:technical} and \Cref{cor:finite-infinite} guarantee that when conditioned to be finite or infinite, this set has the same distribution under $\warb^T$ and $\mathbb Q$. Therefore, given $\mathcal A$, $E_{i-1} = \bigcup_{ e \in B \cap \partial E_{i-1} } K(e^+) \cap (B(e) \backslash \{ e \})$ has the same distribution under $\warb^T$ and $\mathbb Q$.
\end{proof}

\begin{proof}[Proof of \Cref{cor:stochastic-domination}]
Item 1 follows from item 3 of \Cref{thm} because there is a monotone coupling of the bond percolation measures $(\p_p^T)_{p \in [0,1]}$ \cite[Theorem 2.1]{percolation-book}. For item 2, we explicitly modify the procedure from item 3 of \Cref{thm} to construct a monotone coupling $(\omega_{\text{open}}, \omega_{\text{closed}})$ of $(\warb^T \bra{ \; \cdot \; \mid e \in \omega }, \warb^T \bra{\; \cdot \; \mid e \not \in \omega})$. By automorphism-invariance, we can assume $e$ is adjacent to $o$, say $e = \{ o , u \}$.

Sample $\omega_0$ according to $\p_{p_c}^{T\backslash \{e\}}$. Sample a configuration $\eta$ on the vertices of $T$ except $\{o,u\}$ such that each vertex is independently included with probability $\frac{\beta(k-1)-1}{\beta(k-2)}$. Let $U$ be the set of vertices $v \in \eta$ such that there is an edge $f \in T \backslash (\omega \cup \{ e \})$ with $v = f^+$. For each vertex $v \in T$, independently select a neighbour $s(v)$ with $d(e,s(v)) > d(e,v)$ uniformly at random. Let $\gamma(v)$ denote the edges in the path $v,s(v),s(s(v)),\dots$. Let $A_{\text{closed}}$ be a random subset of $\{ o,u \}$ such that each vertex is independently included with probability $\warb^T (o \xleftrightarrow{T \backslash T(e)} \infty \mid e \not \in \omega)$. Let $A_{\text{open}}$ be a random subset of $\{ o, u\}$ containing at most one vertex such that each vertex is included with probability $\warb^T ( o \xleftrightarrow{T \backslash T(e)} \infty \mid e \in \omega)$. By item 3 of \Cref{thm}, $A_{\text{open}}$ is well-defined, and the configurations of the edges in $T \backslash \{ e \}$ given by
\[ \begin{split}
	\omega_{\text{open}} := \omega_0 \cup \bigcup_{v \in U \cup A_{\text{open}} } \gamma(v) \quad\quad \text{and} \quad\quad	\omega_{\text{closed}} := \omega_0 \cup \bigcup_{v \in U \cup A_{\text{closed}}} \gamma(v)\end{split}\]
are distributed according to $\warb^T \bra{ \; \cdot \; \mid e \in \omega }$ and $\warb^T \bra{\; \cdot \; \mid e \not \in \omega}$, respectively.

We automatically have $\p \bra{ u \in A_\text{open} \mid o \in A_{\text{open}} } \leq \p \bra{ u \in A_{\text{closed}} }$. Using \Cref{thm}, we find by direct calculation that $\p \bra{ u \in A_{\text{open}} \mid o \not\in A_{\text{open}} } = \p \bra{ u \in A_{\text{closed}}}$ and $\p \bra{ o \in A_{\text{open}} } \leq \p \bra{ o \in A_{ \text{closed} } }$. So we can couple $A_{\text{open}}$ and $A_{\text{closed}}$ with $A_{\text{open}} \subseteq A_{\text{closed}}$ almost surely. In particular, we can couple $\omega_{\text{open}}$ and $\omega_{\text{closed}}$ with $\omega_\text{open} \subseteq \omega_\text{closed}$ almost surely.
\end{proof}

\section*{Acknowledgements}
We thank Tom Hutchcroft for suggesting this model and for many valuable comments. We also thank Roland Bauerschmidt for helpful feedback on our presentation of an earlier version of these results. We were supported by an EPSRC DTP grant.

\bibliographystyle{ieeetr}
\bibliography{references}

\end{document}